\documentclass[12pt,reqno]{amsart}

\usepackage{amsmath,amsfonts,amssymb,amsthm}
\usepackage[normalem]{ulem}
\usepackage{enumerate}
\usepackage[all]{xy}
\usepackage{color}
\usepackage[colorlinks,citecolor=darkgreen,linkcolor=darkgreen]{hyperref}

\definecolor{darkgreen}{rgb}{0,0.45,0} 
\definecolor{orange}{rgb}{1.0,0.45,0.45} 

\newtheorem{lemma}{Lemma}[section]
\newtheorem{proposition}[lemma]{Proposition}
\newtheorem{corollary}[lemma]{Corollary}
\newtheorem{theorem}[lemma]{Theorem}
\newtheorem{examples}[lemma]{Examples}
\newtheorem{remark}[lemma]{Remark}
\newcommand{\X}{\mathbb{X}}
\newcommand{\C}{\mathbb{C}}
\newcommand{\D}{\mathbb{Y}}
\newcommand{\I}{\mathbb{I}}
\newcommand{\SplExt}{\textnormal{SplExt}}

\newcommand{\aru}{\ar@<0.5ex>}
\newcommand{\ard}{\ar@<-0.5ex>}
\newcommand{\two}{\mathbf{2}}

\begin{document}
\title{Action representability of the category of internal groupoids}
\author{Marino Gran and James Richard Andrew Gray}
\maketitle
\begin{abstract}
 When $\mathbb C$ is a semi-abelian category, it is well known that the category $\mathsf{Grpd}(\mathbb C)$ of internal groupoids in $\mathbb C$ is again semi-abelian. The problem of determining whether the same kind of phenomenon occurs when the property of being semi-abelian is replaced by the one of being action representable (in the sense of Borceux, Janelidze and Kelly) turns out to be rather subtle. In the present article we give a sufficient condition for this to be true: in fact we prove that the category $\mathsf{Grpd}(\mathbb C)$ is a semi-abelian action representable algebraically coherent category with normalizers if and only if $\mathbb C$ is a semi-abelian action representable algebraically coherent category with normalizers. This result applies in particular to the categories of internal groupoids in the categories of groups, Lie algebras and cocommutative Hopf algebras, for instance.
\end{abstract}

\section{Preliminaries}\label{prelims}
In this paper $\mathbb C$ will always denote a semi-abelian category (in the
sense of Janelidze, M\'arki and Tholen \cite{JMT}), usually satisfying some
additional axioms. Recall that a category $\mathbb C$ is \emph{semi-abelian}
if it is
\begin{itemize}
\item finitely complete, finitely cocomplete and pointed, with zero object $0$;
\item (Barr)-\emph{exact} \cite{Barr};
\item (Bourn)-\emph{protomodular} \cite{Bourn}, which in the pointed case can be expressed by the validity of the \emph{Split Short Five Lemma} in $\mathbb C$.
\end{itemize}
There are plenty of interesting algebraic categories which are semi-abelian. For example, any variety of algebras whose algebraic theory has among its operations and identities those of the theory of groups is semi-abelian (see \cite{BJ1} for a precise characterization). As a consequence, the categories $\mathsf{Grp}$ of groups, $\mathsf{Ab}$ of abelian groups, $\mathsf{Rng}$ of (not necessarily unitary) rings, $\mathsf{Lie}_R$ of Lie algebras over a commutative ring $R$, $\mathsf{XMod}$ of crossed modules (of groups), are all semi-abelian categories. In addition any category of compact Hausdorff models of a semi-abelian algebraic theory \cite{BC}, such as the category $\mathsf{Grp}( \mathsf{Comp})$ of compact Hausdorff groups, is semi-abelian, as is also the category $\mathsf{Hopf}_{K,\text{coc}}$ of cocommutative Hopf algebras over a field $K$ \cite{Hopf}. The dual category $\mathsf{Set}_{*}^{\text{op}}$ of the category $\mathsf{Set}_{*}$ of pointed sets is semi-abelian \cite{JMT}. Other examples can be derived from the fact that the category of internal groupoids in a semi-abelian category is semi-abelian \cite{BG2}.

In any semi-abelian category $\C$ there is a natural notion of centrality of arrows \cite{Huq, BG} (in fact weaker assumptions on the base category can be required \cite{BB}, but in this work we shall always ask $\mathbb C$ to be at least semi-abelian). Given two morphisms $f \colon A \rightarrow B$ and $g \colon C \rightarrow B$ with the same codomain, they are said to \emph{commute} in the sense of Huq \cite{Huq} if there is a (necessarily unique) arrow $c \colon A \times C \rightarrow B$ making the diagram commute
\[ \xymatrix{A \ar[r]^-{(1_A,0)} \ar[dr]_f & A \times C \ar@{.>}[d]^{c} &C \ar[l]_-{(0, 1_C)} \ar[dl]^{g} \\
& B,& }
\]
where $(1_A,0)$ and $(0, 1_C)$ are the unique morphisms induced by the universal property of the product $A \times C$.
When this is the case the unique arrow $c \colon A \times C \rightarrow B$ is called the \emph{cooperator} of $f$ and $g$. One usually writes $[f,g]_{\mathrm{Huq}}= 0$, or simply $[f,g] = 0$, when this is the case. Given two subobjects $f \colon A \rightarrow B$ and $g \colon C \rightarrow B$ with the same codomain, the \emph{Huq commutator} $[f,g]$, usually denoted by $[A,C]$ (if there is no risk of confusion), is the smallest normal subobject $D$ of $B$ with the following universal property: in the quotient $\pi \colon B \rightarrow \frac{B}{D}$ the regular images $\pi (A) $ and $\pi (C)$ (of $f \colon A \rightarrow B$ and $g \colon C \rightarrow B$ along $\pi$) commute in the sense above.

Given a morphism $f \colon A\to B$ in a semi-abelian category $\C$ we will denote by $z_f: Z_B(A,f)\to B$ the centralizer of $f$ in $B$, i.e. the terminal object in the category of morphisms that commute with $f$, whenever it exists (see e.g. \cite{BJ, Gray-1}). In this case $z_f$ is always a monomorphism, and we write $Z_B(A,f)$ or $Z_B(A)$ (when there is no risk of confusion) for the corresponding subobject of $B$. For a monomorphism $f: A\to B$ the normalizer of $f$ is the terminal object in the category with objects triples $(N,n,m)$ where $n$ is a normal monomorphism and $m$ a monomorphism such that $mn=f$ \cite{Gray0}.

Recall that a split extension is a diagram in $\C$
\begin{equation}\label{spl-ext}
\xymatrix{
X \ar[r]^{\kappa} & A \aru[r]^{\alpha} & B \aru[l]^{\beta}
}
\end{equation}
where $\kappa$ is the kernel of $\alpha$ and $\alpha \beta = 1_B$. A morphism of split extensions is a diagram in $\C$
\[
\xymatrix{
X\ar[r]^{\kappa}\ar[d]_{u} & A \aru[r]^{\alpha}\ar[d]^{v} & B\aru[l]^{\beta}\ar[d]^{w}\\
X'\ar[r]^{\kappa'} & A' \aru[r]^{\alpha'} & B'\aru[l]^{\beta'}
}
\]
where the top and bottom rows are split extensions (the domain and codomain, respectively), and $v\kappa=\kappa'u$, $v\beta=\beta' w$ and $w\alpha=\alpha'v$. Let us write $\SplExt(\C)$ for the category of split extensions in $\C$, and write $P, K : \SplExt(\C)\to \C$ for the functors sending a split extension to its \emph{codomain} and to the (object part of the) \emph{kernel}, respectively. The category $\C$ can be equivalently defined to be action representable, in the sense of Borceux, Janelidze, Kelly \cite{BJK} when each fiber of the functor $K$ has a terminal object. This means that for each $X$ in $\C$ there exists a split extension
\[
 \xymatrix{
  X \ar[r]^-{k} & [X]\ltimes X \aru[r]^-{p_1} & [X],\aru[l]^-{i}
 }
\]
called the \emph{generic split extension with kernel $X$},
with the universal property that there exists a unique morphism to it from each split extension in the fiber $K^{-1}(X)$, that is, there is a unique morphism, which is the identity on kernels, to it from each split extension \eqref{spl-ext} with kernel $X$:
\[
\xymatrix{
X\ar[r]^{\kappa}\ar@{=}[d] & A \aru[r]^{\alpha}\ar@{.>}[d]^{} & B \aru[l]^{\beta}\ar@{.>}[d]^{}\\
X\ar[r]^-{k} & [X]\ltimes X  \aru[r]^-{p_1} & [X]. \aru[l]^-{i}
}
\]
For instance, in the category $\mathsf{Grp}$ of groups, the generic split extension with kernel a group $X$ is given by the split extension 
\[
 \xymatrix{
  X \ar[r]^-{k} & \mathsf{Aut}(X) \ltimes X \aru[r]^-{p_1} &  \mathsf{Aut}(X),\aru[l]^-{i}
 }
\]
where $\mathsf{Aut}(X)$ is the group of automorphisms of $X$ and the action of $\mathsf{Aut}(X)$ on the group $X$ is given by the evaluation.

More generally, for $X$ an object in $\C$, a split extension in $K^{-1}(X)$ is called \emph{faithful} if there is at most one morphism to it from each split extension in $K^{-1}(X)$. The category $\C$ is called \emph{action accessible} \cite{BJ} if
for each $X$ in $\C$ each split extension in $K^{-1}(X)$ admits a morphism to 
a faithful split extension in $K^{-1}(X)$.
The category $\C$ is \emph{algebraically coherent} \cite{Coherent} when the change of base functors of fibers of $P$ preserve joins. This is equivalent (in the pointed protomodular context) to requiring that for each cospan of monomorphisms of split extensions
\[
\xymatrix{
 X_1\ar[r]^{\kappa_1}\ar[d]_{u_1} & A_1 \aru[r]^{\alpha_1}\ar[d]^{v_1} & B\aru[l]^{\beta_1}\ar@{=}[d]\\
X\ar[r]^{\kappa} & A \aru[r]^{\alpha} & B\aru[l]^{\beta}\\
X_2\ar[r]^{\kappa_2}\ar[u]^{u_2} & A_2 \aru[r]^{\alpha_2}\ar[u]_{v_2} & B\aru[l]^{\beta_2}\ar@{=}[u]\\
}
\]
if the morphisms $v_1$ and $v_2$ are jointly strongly epimorphic in $\C$, then so are the morphisms $u_1$ and $u_2$.
Recall that a semi-abelian algebraically coherent category $\C$ with normalizers has the following properties:
\begin{itemize}
 \item $\C$ is action accessible \cite{BJ} (see also \cite{Gray}), and hence centralizers of normal monomorphisms exist and are normal (Proposition 5.2 of \cite{BournGray2}).
\item Huq commutators distribute over joins of subobjects \cite{Gray2}:
given three subobjects $A_1 \rightarrow C$, $A_2 \rightarrow C$ and $B \rightarrow C$ of the same object $C$ the Huq commutator satisfies the following identity:
$$[A_1 \vee A_2, B] = [A_1 ,B] \vee [ A_2, B].$$
\item The Jacobi identity holds for normal subobjects (Theorem 7.1 \cite{Coherent}): if $K,L,M$ are normal subobject of an object $C$, then
\begin{equation}\label{Jacobi}
[K,[L,M]] \leq [[K,L],M]\vee [[M,K],L]. 
\end{equation}
\item a split extension in $\C$
\[
\xymatrix{
X \ar[r]^{\kappa} & A \aru[r]^{\alpha} & B\aru[l]^{\beta}
}
\]
is faithful if and only if $Z_A(X,\kappa)\wedge B=0$ (see \cite{Bourn2} Corollary 4.1).
\end{itemize}

\section{Reflexive graphs and groupoids}

Recall that a reflexive graph in $\C$ is a diagram
\begin{equation}\label{graph}
      \xymatrix {
    G \ar@<6pt>[r]^{\sigma} \ar@<-6pt>[r]_{\tau} & G_0 \ar[l]|-{e}}
\end{equation}
in $\C$ such that $\sigma e = 1_{G_0} = \tau e$. 
Equivalently, when $\C$ has equalizers, a reflexive graph can be defined as a triple
$(G,s:G\to G,t:G\to G)$ where $st=t$ and $ts=s$. Indeed, the second form is obtained from
the first by setting $s=e\sigma$ and $t=e\tau$. On the other hand the first form
is obtained from the second by choosing $e: G_0 \to G$ to be the equalizer of
$s$ and $1_G$ and constructing $\sigma$ and $\tau$ via the universal property of this
equalizer.

When $\C$ is a semi-abelian action accessible category, an internal groupoid in $\C$ can be equivalently presented as a triple $(G, s \colon G \rightarrow G,  t \colon G \rightarrow G)$ with $st=t$ and $ts=s$ such that, moreover, the commutator of the kernels of $s$ and $t$ is trivial: \begin{equation}\label{commut}
[ \ker (s), \ker (t)] = 0.\end{equation}
This follows from the results in \cite{CPP, BJ} (the Smith commutator and the Huq commutator coincide in this context) and from the fact that 
 $$[\ker (\sigma), \ker (\tau)]= [ \ker (s), \ker (t)]  = 0,$$ 
 since $e$ is a monomorphism. 
  
The category of groupoids in $\C$ is then equivalent to the category whose objects are triples $(G, s \colon G \rightarrow G,  t \colon G \rightarrow G)$ as above with $[\ker (s), \ker (t)]  = 0$, and arrows $$f \colon (G, s \colon G \rightarrow G,  t \colon G \rightarrow G) \rightarrow (H, s' \colon H \rightarrow H,  t' \colon H \rightarrow H)$$ those $f \colon G \rightarrow H$ in $\C$ such that $s'f = fs$ and $t'f = ft$. Observe that this alternative presentation of the notion of internal groupoid can be seen as a generalization of the notion of \emph{$1$-cat group} (in the sense of \cite{Loday}) to the semi-abelian context.

With a slight abuse of notation, since $\C$ will always be assumed to be semi-abelian, from now on we shall write $\mathsf{Grpd}(\C)$ for this latter equivalent category, and also call its objects internal groupoids. We shall also denote by $\ker(s): {}_*G\to G$ and $\ker(t):G_*\to G$, the kernel of $s:G\to G$ and $t:G\to G$, respectively. The notation ${}_*G$ for the domain of the kernel of $s$ intuitively reminds one of the fact that its ``elements'' are the (internal) arrows of $G$ whose ``source'' is the ``zero element in $G$'', whereas the arrows in $G_*$ have as ``target'' this same ``zero element''. We shall also simply write $[{}_*G, G_*]$ to denote the commutator $[\ker(s), \ker (t)]$.

 The remaining part of this section consists largely of a series of lemmas building up to our main
 results: Theorems \ref{thm:min} and \ref{thm:main}, and Corollary \ref{cor:main}. 
Recall that in a pointed protomodular category for a split extension \eqref{spl-ext}
the object $A$ is the join of $X$
and $B$ in $A$. Indeed, if $S$ is a subobject of $A$ containing $X$ and $B$,
then there are monomorphisms $u: X\to S$, $v:B\to S$ and $m: S\to A$ such that
$\kappa =mu$ and $\beta=mv$. This easily implies that $u$ is the 
kernel of $\alpha m$, which is a split epimorphism with section $v$. The split short
five lemma now implies that $m$ is an isomorphism and hence $A\leq S$.
Let us also recall the following known result (see Lemma 2.6 in \cite{CGVdL}):
 \begin{lemma}\label{lift}
 Let $\C$ be a semi-abelian category.
 Consider a split extension as in the bottom row of the diagram 
 \[
\xymatrix{
K\ar@{.>}[r]^{}\ar[d]_k& K \vee Z \aru@{.>}[r]^{}\ar@{.>}[d]^{} & Z \aru@{.>}[l]^{}\ar@{=}[d]^{}\\
X\ar[r]_-{x} & Y  \aru[r]^-{f} & Z \aru[l]^-{s}
}
\]
  in $\C$, with the property that $x k \colon K \rightarrow Y$ is a normal monomorphism. Then this split extension lifts along $k \colon K \rightarrow X$ to yield a normal monomorphism of split extensions, where $K \vee Z$ is the join of the subobjects $K$ and $Z$ of $Y$.
 \end{lemma}
\begin{lemma}\label{lemma: largest sub-reflexive-graph with partial composition}
Let $\C$ be a semi-abelian action accessible category.
For each split extension 
\[
\xymatrix{
X \ar[r]^{\kappa} & A \aru[r]^{\alpha} & B \aru[l]^{\beta}
}
\]
of internal reflexive graphs, there exists a largest sub-reflexive-graph $\tilde B$ of $B$ such that $[{}_*\tilde B,X_*]=0=[\tilde B_*,{}_*X]$ in $\C$.
\end{lemma}
\begin{proof}
We will show that
\[\tilde B = ((Z_A({}_*X,\kappa \ker(s)) \wedge B_*) \vee B_0) \wedge ((Z_A(X_*,\kappa \ker(t)) \wedge  {}_*B) \vee B_0)\]
is the largest sub-reflexive-graph of $B$ satisfying the desired property.
 Since any subobject of $B$ containing $B_0$ (the ``object of objects'' of the reflexive graph $B$) is a sub-reflexive graph of $B$,
 we know that $\tilde B$ is a sub-reflexive-graph of $B$.
 To see that it satisfies the desired property let $Z_1= Z_A({}_*X,\kappa \ker(s))\wedge B_*$, and note that, since ${}_*X=X\wedge {}_*A$, it follows that ${}_*X$ is a normal subobject of $A$ and hence so is $Z_A({}_*X,\kappa \ker(s))$. Accordingly, $Z_1$ is a normal subobject of $B$, and this implies that $(Z_1\vee B_0)_*=Z_1$ - just note that we are in the situation of Lemma \ref{lift}, with $Y=B$, $X= B_*$ and $Z=B_0$. From this we then deduce that
 \[\tilde B_* \leq (Z_1\vee B_0)_* = Z_1 \leq Z_A({}_*X,\kappa\ker(s)).\]
A similar argument shows that ${}_*\tilde B \leq Z_A(X_*,\kappa \ker(t))$. Now, let $B'$ be a sub-reflexive-graph of $B$ satisfying the desired property. Clearly $B'_* \leq B_*$ and $B'_* \leq Z_A({}_*X,\kappa \ker(s))$, and hence $$B'_* \leq Z_A({}_*X,\kappa \ker(s)) \wedge B_*.$$ Since $B'_0 \leq B_0$ and $B' = B'_* \vee B'_0$ (by protomodularity), it follows that $B' = B'_* \vee B'_0 \leq (Z_A({}_*X,\kappa \ker(s)) \wedge B_*) \vee B_0$. By exchanging the roles of $s$ and $t$ of the internal reflexive graphs, we have that $$B' \leq (Z_A(X_*, \kappa \ker(t)) \wedge {}_*B) \vee B_0,$$ from which the claim easily follows.
\end{proof}
\begin{lemma}\label{lemma: commutator commutes with kernel}
Let $\C$ be a semi-abelian action accessible algebraically coherent category.
For each split extension 
\[
\xymatrix{
X \ar[r]^{\kappa} & A \aru[r]^{\alpha} & B \aru[l]^{\beta}
}
\]
of internal reflexive graphs, if $X$ is a groupoid and $$[{}_* B,X_*]=0=[ B_*,{}_*X],$$ then $$[X,[A_*,{}_*A]]=0.$$
\end{lemma}
\begin{proof}
Note that since $\C$ is protomodular we have $A_*=X_*\vee B_*$ and ${}_*A={}_*X\vee{}_*B$. Note that trivially $[X,B_*] \leq X$, but also that $[X,B_*] \leq A_*$ meaning that $[X,B_*] \leq X \wedge A_* = X_*$. We have
\begin{align*}
[X,A_*] &= [X,X_*]\vee [X,B_*]\\
& \leq X_* \vee X_*= X_*.
\end{align*}
Therefore
\begin{align*}
[{}_*A, [X,A_*]] &\leq [{}_*A,X_*]\\
&= [{}_*X,X_*] \vee [{}_*B,X_*]\\
&= 0,
\end{align*}
where the last equality follows from the equality $[{}_*X,X_*]=0$ (since $X$ is a groupoid) and the assumption $[{}_* B,X_*]=0$.

This and its dual (i.e. swapping $s$ and $t$) mean that $$[A_*,[X,{}_*A]] =0=[{}_*A,[X,A_*]].$$ The Jacobi identity \eqref{Jacobi} now implies that 
\[
[X,[A_*,{}_*A]]\leq [[X,A_*],{}_*A]\vee[A_*,[X,{}_*A]]=0.
\]
\end{proof}
\begin{lemma}
\label{lemma: groupoids are extension closed}
Let $\C$ be a semi-abelian action accessible algebraically coherent category.
For each split extension
\[
\xymatrix{
X \ar[r]^{\kappa} & A \aru[r]^{\alpha} & B \aru[l]^{\beta}
}
\]
of internal reflexive graphs. If $X$ and $B$ are groupoids and $$[{}_* B,X_*]=0=[B_*,{}_*X],$$ then $A$ is a groupoid.
\end{lemma}
\begin{proof}
We have
\begin{align*}
[A_*,{}_*A] &= [X_* \vee B_*, {}_* X \vee {}_*B]\\
&= [X_*,{}_*X]\vee [X_*,{}_*B] \vee [B_*,{}_*X]\vee [B_*,{}_*B]\\
&=0.
\end{align*}
\end{proof}
As follows from Theorem 2.1 of \cite{GRAY_2020_b} if $\C$ admits centralizers,
$\I$ is a finite category, and $f:A\to C$ is a morphism in the functor category $\C^\I$, then
the centralizer of $f$ exists. In addition, it follows from the same theorem
that if $g:B\to C$ is the
centralizer of 
$f$ and  $X$ is an object in $\I$, then $B(X)$ together with the
morphism $g_x:B(X)\to C(X)$ can be constructed as follows:

For each
morphism $i:X\to Y$ in $\I$ let $w_i : W_i \to  C(X)$ be the \emph{preimage}
of the centralizer of $f_Y$ along $C(i)$ as displayed in the pullback
\[
\xymatrix{
W_i \ar[r]^{w_i}\ar[d]_{\tilde i} & C(X)\ar[d]^{C(i)}\\
Z_C(Y)(C(X),f_Y) \ar[r]_-{z_{f_Y}} & C(Y).
}
\]
The pair $(B(X),g_X)$ is then the product of the objects $(W_i,w_i)$ in
the comma category $(\C\downarrow C(X))$. Note that if $\I$ is the monoid (considered as a one object category) with identity
element $e$ and generated by $s$ and $t$ satisfying $st=t$ and $ts=s$, then $\C^\I$ is
essentially the category of reflexive graphs in $\C$. Applying the above mentioned construction to this 
special case we obtain:
\begin{lemma}
Let $\C$ be a semi-abelian category admitting centralizers. If $S$ is sub-reflexive-graph of $A$, then the centralizer of $S$ in $A$ has underlying object $Z\wedge s^{-1}(Z)\wedge t^{-1}(Z)$ where $Z$ is the centralizer of the underlying subobject inclusion of $S$ in $A$,  and $s^{-1}(Z)$ and $t^{-1}(Z)$ are the \emph{inverse images} of $Z$ along $s$ and $t$, respectively.
\end{lemma}
\begin{lemma}\label{lemma: description of faithful extensions of graphs}
Let $\C$ be a semi-abelian action accessible algebraically coherent category. 
A split extension
\[
\xymatrix{
X \ar[r]^{\kappa} & A \aru[r]^{\alpha} & B \aru[l]^{\beta}
}
\]
of internal graphs is faithful if and only if  $$Z_A(X,\kappa) \wedge s^{-1}(Z_A(X,\kappa)) \wedge t^{-1}(Z_A(X,\kappa))\wedge B = 0$$ in $\C$.
\end{lemma}
\begin{proof}
Since according to Corollary 2.3 of  \cite{GRAY_2020_b} the category of internal reflexive graphs in $\C$ is action accessible as soon as $\C$ is, the claim follows from the previous lemma via the last bullet of Section \ref{prelims}.
\end{proof}
\begin{lemma}\label{lemma: codomain of faithful extension groupoid}
Let $\C$ be a semi-abelian action accessible algebraically coherent category. 
For each faithful split extension
\[
\xymatrix{
X \ar[r]^{\kappa} & A \aru[r]^{\alpha} & B \aru[l]^{\beta}
}
\]
of internal reflexive graphs, $B$ is a groupoid if and only if $$[X,[B_*,{}_*B]]=0$$ in $\C$.
\end{lemma}
\begin{proof}
If $B$ is a groupoid then this is trivially the case. The converse follows from Lemma \ref{lemma: description of faithful extensions of graphs} and the fact that $[B_*,{}_*B]$ is always in $s^{-1}(Z_A(X,\kappa))$ and $t^{-1}(Z_A(X,\kappa))$ (because $s([B_*,{}_*B])=0$ and similarly for $t$).
\end{proof}
\begin{theorem}\label{thm:min}
Let $\C$ be a semi-abelian action accessible algebraically coherent category. 
 For each faithful split extension
\[
\xymatrix{
X \ar[r]^{\kappa} & A \aru[r]^{\alpha} & B \aru[l]^{\beta}
}
\]
of reflexive graphs with $X$ a groupoid, there exists a largest sub-split-extension of groupoids with kernel $X$.
\end{theorem}
\begin{proof}
By Lemma \ref{lemma: largest sub-reflexive-graph with partial composition} there is a largest sub-reflexive-graph $\tilde B$ of $B$ such that $[{}_*\tilde B,X_*]=0=[\tilde B_*,{}_*X]$. We will prove that the split extension at the top of the diagram
\[
\xymatrix{
X\ar@{=}[d] \ar[r]^{\tilde \kappa} & \tilde A \ar[d] \aru[r]^{\tilde \alpha} & \tilde B \ar[d] \aru[l]^{\tilde \beta}\\
X \ar[r]^{\kappa} & A \aru[r]^{\alpha} & B \aru[l]^{\beta}
}
\]
obtained by pulling back along $\tilde B \to B$ is the desired split extensions of groupoids. According to Lemma \ref{lemma: commutator commutes with kernel} $[X,[\tilde A_*,{}_*\tilde A]] =0$ which means that $[X,[\tilde B_*,{}_*\tilde B]]=0$. Therefore, since a sub-split-extension of a faithful extension is faithful, it follows from Lemma \ref{lemma: codomain of faithful extension groupoid} that $\tilde B$ is a groupoid. The final claim then follows from Lemma \ref{lemma: groupoids are extension closed}.
\end{proof}

We will also need the following proposition which shows that a coreflective
subcategory closed under certain limits admits generic split extensions
whenever the category it is coreflective in does. Recall that a functor between pointed categories is \emph{protoadditive} \cite{EG} when it preserves split exact sequences.
 \begin{proposition}
Let $\X$ and $\D$ be semi-abelian categories, and let \\ $I : \X\to \D$ be a full and faithful
protoadditive functor with right adjoint $R\colon \D \to \X$. If $\D$ has generic split
extensions, then so does $\X$.
\end{proposition}
\begin{proof}
 Suppose $X$ is an object in $\X$ and suppose that
 \[
  \xymatrix{
   I(X) \ar[r]^-{k} & [I(X)]\ltimes I(X) \ar@<0.5ex>[r]^-{p_1} & [I(X)] \ar@<0.5ex>[l]^-{i}
  }
 \]
 is the generic split extension with kernel $I(X)$ in $\D$. Let $\eta$ be
 the unit of the adjunction $I \dashv R$ which is an isomorphism. The claim
 now follows by observing that (i) the lower part of \eqref{coref_2} (below) is a
 split extension; (ii) for each split extension \eqref{spl-ext} the upper part
 of \eqref{coref_1} (below) is a split extension, and the
 adjunction produces a bijection between morphisms of split extensions of
 the form
 \begin{equation}
  \label{coref_1}
  \vcenter{
  \xymatrix{
   I(X)\ar@{=}[d]\ar[r]^{I(\kappa)} & I(A)\ar@{-->}[d] \ar@<0.5ex>[r]^{I(\alpha)} & I(B)\ar@<0.5ex>[l]^{I(\beta)}\ar@{-->}[d]\\
   I(X) \ar[r]^-{k} & [I(X)]\ltimes I(X) \ar@<0.5ex>[r]^-{p_1} & [I(X)] \ar@<0.5ex>[l]^-{i}
  }}
 \end{equation}
 in $\D$ 
 and morphisms of split extensions of the form
 \begin{equation}
  \label{coref_2}
  \vcenter{
  \xymatrix@=30pt{
   X\ar@{=}[d]\ar[r]^{\kappa} & A\ar@{-->}[d] \ar@<0.5ex>[r]^{\alpha} & B\ar@<0.5ex>[l]^{\beta}\ar@{-->}[d]\\
   X \ar[r]^-{R(k)\eta_X} & R([I(X)]\ltimes I(X)) \ar@<0.5ex>[r]^-{R(p_1)} & R([I(X)]) \ar@<0.5ex>[l]^-{R(i)}
  }
 }
 \end{equation}
 in $\X$.
\end{proof}

  \begin{theorem}\label{thm:main}
  A category $\C$ is a semi-abelian action representable
 algebraically coherent category with normalizers if and only if the category
 $\mathsf{Grpd} ({ \C})$ of internal groupoids in $\C$ is a semi-abelian
 action representable algebraically coherent category with normalizers.
\end{theorem}
\begin{proof}
For the ``if'' part suppose that $\mathsf{Grpd}(\C)$ is a
 semi-abelian algebraically coherent category with normalizers. 
 Noting that functor from $\C$ to
 $\mathsf{Grpd}({\C})$, sending an object in $\C$ to the discrete groupoid in
 $\mathsf{Grpd}(\C)$ embeds $\C$ as a full reflective and coreflective subcategory
 of $\mathsf{Gpd}(\C)$ (closed in $\mathsf{Grpd}(\C)$ under quotients and subobjects), it follows that $\C$ is a semi-abelian
 algebraically coherent category with normalizers. Action representability
 now follows from the previous proposition.
 For the ``only if'' part suppose that $\C$ is a semi-abelian action representable algebraically
 coherent category with normalizers.
 It follows from \cite{Gray0} that the category of reflexive graphs being a functor category is action representable. For a groupoid $X$, applying the previous theorem to the generic split extension of reflexive graphs with kernel $X$ it easily follows that the largest sub-split extension of groupoids with kernel $X$ is the generic split extension of groupoids with kernel $X$. The fact that $\mathsf{Grpd} ({\C})$ is semi-abelian when $\C$ is semi-abelian follows from Lemma $4.1$ in \cite{BG2}. Proposition 4.18 of \cite{Coherent} now tells us that $\mathsf{Grpd }(\C)$ is algebraically coherent, and hence it remains to show that $\mathsf{Grpd}(\C)$ has normalizers.
 However, by Corollary 2.3 of \cite{GRAY_2020_b}
 the category $\mathsf{RG}(\C)$ of reflexive graphs in $\C$ has normalizers and hence so does
 $\mathsf{Grpd} (\C)$
 being closed under subobjects and finite limits in $\mathsf{RG}(\C)$.
\end{proof}

\begin{remark} Note that the ``only if'' part of the above theorem is known
in the special case when $\C$ is the category of groups. This goes back to the work of
K. Norrie \cite{Norrie} 
whose \emph{actors} of crossed modules of groups
are essentially the same, as shown by P. Ramasu in \cite{Ramasu},
as split extension classifiers in the category of internal groupoids in the
category of groups. Let us also mention that D. Bourn has defined action
groupoids whose existence are equivalent to the existence of generic split
extensions in the pointed protomodular context, and has shown that each
category of groupoids with ``fixed object of objects'' admits action groupoids
\cite{Bourn1}.
\end{remark}

Recall that, for a non-negative integer $n$, the category $\mathsf{Grpd}^{n}(\C) $ of $n$-fold internal groupoids
can be thought of as the category of internal groupoids in $\mathsf{Grpd}^{n-1}(\C)$ when $n>0$, and be identified with 
$\C$ when $n=0$. As an immediate corollary of the previous theorem we obtain:
\begin{corollary}\label{cor:main}
 If $\C$ is a semi-abelian action representable algebraically coherent category with normalizers, then so is the category $\mathsf{Grpd}^n ({\mathbb C})$ of $n$-fold internal groupoids in $\C$.
\end{corollary}
\begin{examples}
\emph{The results in this article apply to some important algebraic categories, such as the categories $\mathsf{Grpd}(\mathsf{Grp})$, $\mathsf{Grpd}(\mathsf{{Lie}_R)}$, or $\mathsf{Grpd}(\mathsf{{Hopf}_{K,coc})}$, of internal groupoids in the categories of groups, Lie algebras over a commutative ring $R$, or cocommutative Hopf algebras over a field $K$, respectively. For the fact that $\mathsf{{Hopf}_{K,coc}}$ is an action representable semi-abelian category the reader is referred to \cite{Hopf}, whereas the fact that it is algebraically coherent is explained in Example 4.6 in \cite{Coherent}.  To see that $\mathsf{{Hopf}_{K,coc}}$ has normalizers recall that:
\begin{enumerate}[(a)]
\item $\mathsf{{Hopf}_{K,coc}}$ is equivalent to the category  $\mathsf{Grp}(\mathsf{Coalg}_{K,coc})$ of internal groups in the finitely complete cartesian closed category $\mathsf{Coalg}_{K,coc}$ of cocommutative coalgebras over $K$;
\item  for a cartesian closed category $\X$ with finite limits: (i) $\mathsf{Grp}(\X^\two)\cong(\mathsf{Grp}(\X))^\two$, (ii) $\X^\two$ is cartesian closed;
\item  the category of internal groups in a cartesian closed category with finite limits is action representable as soon as it is semi-abelian \cite{BJK};
\item  a semi-abelian category is action representable and admits normalizers if and only if its category of morphisms is action representable \cite{Gray0}.
\end{enumerate}
 Using the equivalence between the categories of internal groupoids and (internal) crossed modules \cite{Janelidze}, it follows that the categories of crossed modules of groups, $n$-cat groups \cite{Loday}, crossed modules of Lie algebras, crossed $n$-cubes of Lie algebras \cite{Ellis}, and ${\mathsf{cat}}^1$-cocommutative Hopf algebras \cite{Vilaboa, Hopf} are all algebraically coherent action representable semi-abelian categories with normalizers.}
\end{examples}

\bibliographystyle{plain}

\begin{thebibliography}{10}
\bibitem{Barr} M. Barr, Exact categories, Lecture Notes in Math. 236, Springer, Berlin, 1971, 1--120.

\bibitem{BB} F. Borceux and D. Bourn, Mal'cev, protomodular, homological and semi-abelian categories, Kluwer Academic Publishers, 2004.

\bibitem{BC} F. Borceux and M.M. Clementino, \emph{Topological semi-abelian algebras}, Adv. Math. 190, no. 2, 2005, 425--453.

\bibitem{BJK} F.~Borceux, G.~Janelidze, and G.~M. Kelly, \emph{Internal object actions}, Comment. Math. Univ. Carolin. 46 no. 2, 2005, 235--255.

\bibitem{Bourn} D. Bourn, \emph{ Normalization equivalence, kernel equivalence and affine categories,} Category theory (Como, 1990), 43--62, Lecture Notes in Math., 1488, Springer, Berlin, 1991.


\bibitem{Bourn1} D. Bourn, \emph{Action groupoid in protomodular categories,} Theory Appl. Categ. 16, no. 2, 2006, 46--58.

\bibitem{Bourn2} D. Bourn, \emph{Centralizer and faithful groupoid}, J. Algebra 328, no. 1, 2011, 43--76.


\bibitem{BG} D. Bourn and M. Gran, \emph{Centrality and normality in protomodular categories}, Theory Appl. Categ. 9,  no. 8, 2002, 151--165.

 \bibitem{BG2}  D. Bourn and M. Gran, \emph{Central extensions in semi-abelian categories}, J. Pure Appl. Algebra 175, 2002, 31-44.

 \bibitem{BournGray2} D. Bourn and J. R. A. Gray, \emph{Normalizers and split extensions}, Appl. Categ. Structures 23, no. 6, 2015, 753-776. 
 
 \bibitem{BJ1} D. Bourn and G. Janelidze, \emph{Characterization of protomodular varieties of universal algebras},  Theory Appl. Categ. 11,  no. 6, 2003, 143--147.

\bibitem{BJ} D. Bourn and G. Janelidze, \emph{Centralizers in action accessible categories}, Cah. Topol. G\'eom. Diff\'er. Cat\'eg. 50, no. 3, 2009, 211--232.

\bibitem{CPP} A. Carboni, M.C. Pedicchio, N. Pirovano, \emph{Internal graphs and internal groupoids in Mal'cev categories}, Category theory 1991 (Montreal, PQ, 1991), 97-109, CMS Conf. Proc., 13, Amer. Math. Soc., Providence, RI, 1992.

\bibitem{CGVdL} A. S. Cigoli, J. R. A. Gray and T. Van der Linden,  \emph{On the normality of Higgins commutators}, J. Pure Appl. Algebra 219, no. 4, 2015, 897-912.

\bibitem{Coherent} A. S. Cigoli, J. R. A. Gray and T. Van der Linden,  \emph{Algebraically coherent categories}, Theory Appl. Categ. 30, no. 12, 2015, 1864-1905.

\bibitem{Ellis} G. J. Ellis, \emph{Higher dimensional crossed modules of algebras}, J. Pure Appl. Algebra 52, 1998, 277-282.

\bibitem{EG}  T. Everaert and M. Gran, \emph{Protoadditive functors, derived torsion theories and homology}, J. Pure Appl. Algebra 219, no. 8, 2015, 3629-3676.

\bibitem{Vilaboa}  J.M. Fern\'andez Vilaboa, M.P. L\'opez L\'opez and E. Villanueva Novoa, \emph{Cat$^1$-Hopf Algebras and Crossed Modules},  Comm. Algebra 35, no. 1, 2006, 181-191. 

\bibitem{Hopf} M. Gran, F. Sterck and J. Vercruysse, \emph{A semi-abelian extension of a theorem by Takeuchi}, J. Pure Appl. Algebra 223, no. 10, 2019, 4171-4190.

\bibitem{Gray-1} J. R. A. Gray \emph{Algebraic exponentiation in general categories}, Appl. Categ. Structured 20, no. 6, 2012, 543-567.
\bibitem{Gray0} J. R. A. Gray, \emph{Normalizers, centralizers and action representability in semi-abelian categories}, Appl. Categ. Structures 22, no. 5-6, 2014, 981-1007.
\bibitem{Gray} J. R. A. Gray, \emph{Normalizers, centralizers and action accessibility}, Theory Appl. Categ. 30, no. 12. 2015, 410-432.

\bibitem{Gray2} J. R. A. Gray, \emph{A note on the distributivity of the Huq commutator over finite joins}, Appl. Categ. Structures 22, no. 12, 2014, 305-310.
G
\bibitem{GRAY_2020_b} J. R. A. Gray, \emph{A note on images cover relations}, preprint, arXiv:2011.06903, 2020.

\bibitem{Huq} S.A. Huq. \newblock \emph{Commutator, nilpotency, and solvability in categories,} Quart. J. Math. Oxford Ser.(2) 19, 1968, 363--389.

\bibitem{Janelidze} G. Janelidze, \emph{Internal crossed modules}, Georgian Math. J. 10, no. 1, 2003, 99--114.

\bibitem{JMT} G. Janelidze, L. M\'arki and W. Tholen, \emph{Semi-abelian categories}, J. Pure Appl. Algebra 168, no. 2-3, 2002, 367--386.

\bibitem{Loday}  J.-L. Loday, \emph{Spaces with finitely many nontrivial homotopy groups}, J. Pure Appl. Algebra 24, no. 2, 1982, 179--202.

\bibitem{Norrie} K. Norrie, \emph{Actions and automorphisms of crossed modules,} Bull. Soc. Math. France 118, no. 2, 1990, 129--146.


\bibitem{Ramasu} P. Ramasu, \emph{A note on internal object action representability of 1-cat groups and crossed modules,} Theory Appl. Categ. 34 no. 36, 2019, 1165--1178.

\end{thebibliography}

\end{document}